\documentclass[preprint]{elsarticle}
\usepackage{amsmath, amsfonts, amsthm, mathrsfs, amssymb, fancyhdr, graphicx, cancel, multicol, graphicx, subfig, enumitem, pgffor, enumitem}

\newtheorem{thm}{Theorem}
\newtheorem{lem}{Lemma}
\newtheorem*{def1}{Definition}
\newtheorem*{conj}{Conjecture}
\newtheorem{cor}{Corollary}

\begin{document}

\title{Deep Holes in Reed-Solomon Codes Based on Dickson Polynomials}

\author{Matt Keti}
\ead{mketi@uci.edu}

\author{Daqing Wan}
\ead{dwan@math.uci.edu}

\address{Department of Mathematics, University of California Irvine, CA 92697-3875, USA}

\begin{abstract}
For an $[n,k]$ Reed-Solomon code $\mathcal{C}$, it can be shown that any received word $r$ lies a distance at most $n-k$ from $\mathcal{C}$, denoted $d(r,\mathcal{C})\leq n-k$. Any word $r$ meeting the equality is called a deep hole. Guruswami and Vardy (2005) showed that for a specific class of codes, determining whether or not a word is a deep hole is NP-hard. They suggested passingly that it may be easier when the evaluation set of $\mathcal{C}$ is large or structured. Following this idea, we study the case where the evaluation set is the image of a Dickson polynomial, whose values appear with a special uniformity. To find families of received words that are not deep holes, we reduce to a subset sum problem (or equivalently, a Dickson polynomial-variation of Waring's problem) and find solution conditions by applying an argument using estimates on character sums indexed over the evaluation set.
\end{abstract}

\begin{keyword}
character sum \sep deep hole \sep Dickson polynomial \sep Reed-Solomon code \sep subset sum \sep Waring's problem
\end{keyword}

\maketitle

\section{Introduction}

Reed-Solomon error-correcting codes are used routinely in technological applications when there is a risk for transmitted data to be lost or corrupt. The classical set-up fixes a finite field $\mathbb{F}_q$, message block length $k$, and subset of $\mathbb{F}_q$ of size $n>k$ denoted $D=\{x_1, x_2,\ldots,x_n\}$. $D$ is often referred to as the evaluation set with typical choices $D=\mathbb{F}_q$ or $\mathbb{F}_q^*$. A message $(m_0, m_1,\ldots, m_{k-1})$ is represented by the polynomial $m(x)=m_0+m_1x+\ldots+m_{k-1}x^{k-1}$. The message is encoded by calculating $(m(x_1),m(x_2),\ldots,m(x_n))$, called a codeword. The set of all possible encoded messages is defined as the codebook and is denoted by $\mathcal{C}$.

Let $a=(a_1,a_2,\ldots, a_n)$ and $b=(b_1, b_2, \ldots,b_n)$ be two words. Define the Hamming distance $d(a,b)$ as the number of coordinates in which $a$ and $b$ differ. The distance between a word $u$ and the codebook $\mathcal{C}$ is defined as $d(u,\mathcal{C})=\min_{v\in \mathcal{C}} d(u,v)$. It is well known that for Reed-Solomon codes, $d(u,\mathcal{C})\leq n-k$ for any received word $u$. In studying the error-correcting capacity for Reed-Solomon codes, Guruswami and Vardy in \cite{GV} found that for a special family of codes with a small evaluation set, determining whether or not $d(u,\mathcal{C})=n-k$ for a given $u$ is NP-hard. They called any word $u$ satisfying the equality a deep hole and suggested in passing that finding deep holes might be easier when the evaluation set is large. We will further investigate the problem of finding deep holes.

\section{Overview of Previous Work}

One way to measure $d(u,\mathcal{C})$ is to run Lagrange Interpolation on the word $u=(u_1, \cdots, u_n)$ to get a 
fitted polynomial $u(x)$ satisfying $u(x_i) =u_i$ for all $1\leq i\leq n$. Then, if $\deg u(x)\leq k-1$, then $u$ is a codeword and $d(u,\mathcal{C})=0$. Otherwise, $k\leq \deg u(x)\leq n-1$, and Li and Wan in \cite{LW2} gave the bound
\begin{equation*}
n-\deg u(x)\leq d(u,\mathcal{C})\leq n-k
\end{equation*}

\noindent which shows that if $\deg u(x)=k$, then $u$ is automatically a deep hole. Many of the results toward the deep hole problem are geared toward examining families of words by degree.

\subsection{For $D=\mathbb{F}_q$}

The choice of $D=\mathbb{F}_q$ is referred to as a standard Reed-Solomon code. Cheng and Murray in \cite{CM} searched for deep holes in this context, and conjectured that the only deep holes were those satisfying $\deg u(x)=k$.  More precisely,

\begin{conj}[Cheng-Murray]
All deep holes for standard Reed-Solomon codes are those words $u$ satisfying $\deg u(x)=k$.
\end{conj} 

\noindent They weren't able to prove this, but they were able to reduce the problem to finding a rational point on an algebraic hypersurface to derive the first result on deep holes for Reed-Solomon codes over prime field $\mathbb{F}_p$:

\begin{thm}[Cheng-Murray]
Let $p$ be a prime and $1<k<p^{1/4-\varepsilon}$ be a positive integer. Consider the standard Reed-Solomon code $\mathcal{C}_p$, and let $u$ be a received word and $u(x)$ be its interpolated polynomial. If the degree of $u(x)$ satisfies
\begin{equation*}
k<\deg u(x)<k+p^{3/13-\varepsilon}
\end{equation*}
then $u$ is not a deep hole.
\end{thm}

\noindent In one of the newest papers, Cheng, Li, and Zhuang \cite{CLZ} were able to resolve in some cases the conjecture over $\mathbb{F}_p$ using the concept of deep hole trees.

\begin{thm}[Cheng-Li-Zhuang]
Let $p>2$ be a prime number, $k\geq \frac{p-1}{2}$, $D=\{\alpha_1, \alpha_2, \ldots, \alpha_n\}$ with $k<n\leq p$. The only deep holes of $\mathcal{C}_p$ are generated by functions which are equivalent to the following:
\begin{equation*}
f(x)=x^k, \hspace{1cm} f_\delta (x)=\frac{1}{x-\delta}
\end{equation*}
where $\delta \in \mathbb{F}_p \backslash D$. Here, two functions $f(x)$ and $g(x)$ are equivalent if and only if there exists $a\in \mathbb{F}_p^*$ and $h(x)$ with degree less than $k$ such that
\begin{equation*}
g(x)=a f(x) + h(x)
\end{equation*}
\end{thm}

\noindent Over a general finite field $\mathbb{F}_q$, they also showed

\begin{thm}
Given a finite field $\mathbb{F}_q$ with characteristic $p>2$, if $k+1\leq p$ or $3\leq q-p+1\leq k+1\leq q-2$, then the Cheng-Murray conjecture is true.
\end{thm}

\noindent Further evidence in favour of the conjecture has been provided by Li-Wan \cite{LW2}, Liao \cite{L}, Zhu-Wan \cite{WZ}, using polynomial congruences and estimates of character sums, Cafure-Matera-Privitelli \cite{CMP} using algebraic geometry, and Li-Zhu \cite{LZ} by explicitly counting solutions to certain polynomial equations. To list the one by Li-Zhu,

\begin{thm}[Li-Zhu]
Let $\mathcal{C}_q$ be a standard Reed-Solomon code, and $u$ a received word represented by the polynomial $u(x)=x^{k+1}-bx^k+v(x)$, where $\deg(v(x))\leq k-1$. Then
\begin{equation*}
d(u,\mathcal{C}_q)=\left\{
     \begin{array}{ll}
       q-k &  b=0, p=2, k=1 \text{ or } q-3 \\
       q-k-1 & \text{otherwise}
     \end{array}
   \right.
\end{equation*}
\end{thm}

\noindent They also proved a similar distance formula for words of the form $u(x)=x^{k+2}-bx^{k+1}+cx^k+v(x)$, showing that such words are not deep holes either.

\subsection{For $D=\mathbb{F}_q^*$}

The choice of $D=\mathbb{F}_q^*$ is referred to as a primitive Reed-Solomon code. Wu and Hong studied this in \cite{WH} and showed

\begin{thm}[Wu-Hong]\label{wuHong}
Let $\mathcal{C}$ be a primitive Reed-Solomon code over $\mathbb{F}_q$ with $q\geq 4$ and $2\leq k\leq q-2$. Then polynomials of the form $u(x)=ax^{q-2}+v(x)$ with $a\not=0$, where $\deg v(x)\leq k-1$, represent deep holes for $\mathcal{C}$.
\end{thm}

\noindent (This was later generalised to any $D\ne \mathbb{F}_q$ by Wu and Hong \cite{WH2}, and separately by Zhang, Fu, and Liao \cite{FLZ}.) With this new family of deep holes, they revised the Cheng-Murray conjecture:

\begin{conj}[Wu-Hong]
All deep holes for primitive Reed-Solomon codes are those words $u$ represented by $u(x)=ax^k+v(x)$ or $ax^{q-2}+v(x)$ where $\deg(v(x))\leq k-1$.
\end{conj} 

\noindent Cheng, Li, and Zhuang in the previously referenced \cite{CLZ} were able to resolve this under some conditions.

\begin{thm}
Given a finite field $\mathbb{F}_q$ with characteristic $p>2$, if $3\leq k< \frac{\sqrt{q}+1}{4}$ or $3\leq k < \frac{p}{45}$ when $q=p$ is prime, then the Wu-Hong conjecture is true.
\end{thm}

\noindent Li and Zhu \cite{LZ} and Zhang, Fu, and Liao \cite{FLZ} found families of words of degree $k+1$ and $k+2$ that were not deep holes. The latter authors also studied codes over fields of characterstic two, finding yet another family of deep holes.

\begin{thm}[Zhang-Fu-Liao]
Let $q>4$ be a power of 2 and let $\mathcal{C}$ be a Reed-Solomon code over $\mathbb{F}_q$ with $D=\mathbb{F}_q^*$ or $D=\mathbb{F}_q^*/\{1\}$ and $k=q-4$. If $a\ne 0$, then polynomials of the form $u(x)=ax^{q-3}+v(x)$, where $\deg v(x)\leq k-1$, represent deep holes for $\mathcal{C}$.
\end{thm}

\noindent This example shows that the deep hole problem in characteristic two, the most important setting for applications, may be very complicated.

\subsection{For other choices of $D$}

In a previous paper with Zhu \cite{KWZ}, we proved a scheme to find families of words with degree slightly larger than $k$ that were not deep holes, if a certain character sum depending on $D$ could be estimated.

\begin{thm}
Let $\mathcal{C}$ be the generalised Reed-Solomon code over $\mathbb{F}_q$ using the evaluation set $D$. Let $u$ be a received word. Suppose we can write
\begin{equation*}
\left(\frac{w(x_1)}{h(x_1)},\frac{w(x_2)}{h(x_2)},\ldots,\frac{w(x_{|D|})}{h(x_{|D|})} \right)=u
\end{equation*}
for some $h(x)\in \mathbb{F}_q[x]$, with no roots in $D\cup 0$, and $\deg h(x)+k\leq \deg w(x)\leq |D|-1$. Let $m$ be the smallest such degree of $w(x)$. Let $1\leq r\leq d:=m-k\leq |D|-k-1$. If the bound
\begin{equation*}
\left|\sum_{a\in D} \chi(1-ax)\right|\leq K q^{1/2}
\end{equation*}

\noindent is true over all nontrivial characters $\chi:(\mathbb{F}[x]/(\bar{h}(x)))^*\to \mathbb{C}^*$ with $\chi(\mathbb{F}_q^*)=1$ for some $K\geq d$ and 
$\bar{h}(x)=x^{m-k+1}h(1/x)$, there are positive constants $c_1$ and $c_2$ such that if
\begin{equation*}
d\leq K<c_1 \frac{|D|}{q^{1/2}}\ ,\ \left(\frac{d+r}{2}+1\right) \log_2 q < k <c_2|D|
\end{equation*}
\noindent then $d(u,\mathcal{C})\leq |D|-k-r$.
\end{thm}

\noindent We were able to derive a suitable bound taking $D$ as any subgroup of $\mathbb{F}_q^*$:

\begin{thm}
Let $\mathcal{C}$ be the Reed-Solomon code over $\mathbb{F}_q$ using the evaluation set $D=(\mathbb{F}_q^*)^{\frac{q-1}{\ell}}$ of size $\ell$. Let $r\geq 1$ be an integer and $u$ a received word with interpolated polynomial $u(x)$ such that $r\leq d:= \deg(u(x))-k\leq q-2-k$. There are positive constants $c_1$ and $c_2$ such that if
\begin{equation*}
d<c_1 \frac{\ell}{q^{1/2}}\ ,\ \left(\frac{d+r}{2}+1\right) \log_2 q < k <c_2\ell
\end{equation*}
\noindent then $d(u,\mathcal{C})\leq \ell-k-r$.
\end{thm}

\noindent This shows that some progress can be made on classifying deep holes when $D$ is relatively small.

\section{New Results}

Much of the previous work took Reed-Solomon codes with $D=\mathbb{F}_q$, $\mathbb{F}_q^*$, a subgroup of $\mathbb{F}_q^*$, or a very large subset of $\mathbb{F}_q$. 
Any subset $D$ of $\mathbb{F}_q$ can be represented as the image of a polynomial $f(x)$ over $\mathbb{F}_q$. Previous investigations take the polynomial 
$f(x)$ to be the linear polynomial $x$ or the monomial $x^{(q-1)/\ell}$. For a more general polynomial, the problem becomes much harder, as even computing the 
cardinality of the image $f(\mathbb{F}_q)$ is complicated, see \cite{CHW} for more detail.

Here we study the case where $D$ is slightly more general - it will be the image of a Dickson polynomial over $\mathbb{F}_q$, which is defined as follows:

\begin{def1}[Dickson Polynomial]
Let $n$ be a positive integer and $a\in \mathbb{F}_q$. The Dickson polynomial of degree $n$ is defined as
\begin{equation*}
D_n(x, a) = \sum_{i=0}^{\lfloor n/2 \rfloor} \frac{n}{n-i}\binom{n-i}{i} (-a)^i x^{n-2i}. 
\end{equation*}
\end{def1}

\noindent Note that for $a=0$, $D_n(x,0)=x^n$, so we see that Dickson polynomials are a sort of generalisation of monomials. Of particular use to us is the size of the image of these polynomials, also known as the value set. A simple fact for the monomial $D_n(x,0)=x^n$ is that the image of the map $D_n: \mathbb{F}_q^*\to \mathbb{F}_q^*$ has size $q-1$ if $\gcd(n, q-1)=1$ and size $(q-1)/\ell$ if $\gcd(n, q-1)=\ell$. In the first case, the map is $1$ to $1$; in the latter case, the map is $\ell$ to $1$. It turns out an analogous preimage-counting statement holds when $a\ne 0$. Chou, Mullen, and Wassermann in \cite{CMW} used a character sum argument to calculate

\begin{thm}\label{dicksonTheorem}
Let $n\geq 2$ and $a\in \mathbb{F}_q^*$. If $q$ is even, then $|D_n^{-1}(D_n(x_0,a))| =$
\begin{equation*}
 \left\{
     \begin{array}{cl}
       \gcd(n, q-1) & \text{ if condition A holds} \\
       \gcd(n, q+1) & \text{ if condition B holds} \\
       \dfrac{\gcd(n,q-1)+\gcd(n,q+1)}{2} & \text{$D_n(x_0,a)=0$}, 
     \end{array}
   \right.
\end{equation*}
where `condition A' holds if $x^2+x_0x+a$ is reducible over $\mathbb{F}_q$ and $D_n(x_0,a)\ne 0$; `condition B' holds if $x^2+x_0x+a$ is irreducible over $\mathbb{F}_q$ and $D_n(x_0,a)\ne 0$. \newline

\noindent If $q$ is odd, let $\eta$ be the quadratic character of $\mathbb{F}_q$. If $2^r||(q^2-1)$ then \\ $|D_n^{-1}(D_n(x_0,a))|=$
\begin{equation*}
\left\{
     \begin{array}{cl}
       \gcd(n, q-1) & \text{if } \eta(x_0^2-4a)=1  \\
       \ & \text{ and } D_n(x_0,a)\ne \pm 2a^{n/2} \\
       \gcd(n, q+1) & \text{if } \eta(x_0^2-4a)=-1 \\
       \ & \text{ and } D_n(x_0,a)\ne \pm 2a^{n/2} \\
       \dfrac{\gcd(n, q-1)}{2} & \text{if } \eta(x_0^2-4a)=1 \text{ and condition C holds} \\
       \ & \ \\
       \dfrac{\gcd(n, q+1)}{2} & \text{if } \eta(x_0^2-4a)=-1 \text{ and condition C holds} \\
       \ & \ \\
       \dfrac{\gcd(n,q-1)+\gcd(n,q+1)}{2} & \text{otherwise}, \\
     \end{array}
   \right.
\end{equation*}

\noindent where `condition C' holds if
\begin{equation*}
2^t||n \text{ with } 1\leq t\leq r-1, \eta(a)=-1, \text{ and } D_n(x_0,a)=\pm 2a^{n/2}
\end{equation*}
or
\begin{equation*}
2^t||n \text{ with } 1\leq t\leq r-2, \eta(a)=1, \text{ and } D_n(x_0,a)=- 2a^{n/2}
\end{equation*}
\end{thm}

\noindent They also showed an explicit formula for the size of the value set of $D_n(x,a)$, denoted $|V_{D_n(x,a)}|$.

\begin{thm}\label{dicksonTheorem2}
Let $a\in \mathbb{F}_q^*$. If $2^r||(q^2-1)$ and $\eta$ is the quadratic character on $\mathbb{F}_q$ when $q$ is odd, then
\begin{equation*}
|V_{D_n(x,a)}| = \frac{q-1}{2\gcd(n,q-1)}+\frac{q+1}{2\gcd(n,q+1)}+\delta
\end{equation*}
where
\begin{equation*}
\delta = \left\{
     \begin{array}{cl}
       1 & \text{if $q$ is odd, $2^{r-1}||n$ and $\eta(a)=-1$} \\
       \dfrac{1}{2} & \text{if $q$ is odd, $2^{t}||n$ with $1\leq t\leq r-2$} \\
       0 & \text{otherwise}
     \end{array}
   \right.
\end{equation*}
\end{thm}

\noindent These results together with a character sum argument lead us to our main theorem.

\subsection{Main Theorem}

\begin{thm}
Let $\mathcal{C}$ be the Reed-Solomon code over $\mathbb{F}_q$ with message length $k$, using the evaluation set $D=\{D_n(x,a)\ |\ x\in \mathbb{F}_q\}$, for $a\in \mathbb{F}_q^*$. Let $u$ be a received word and $u(x)$ its interpolated polynomial with $\deg u(x)=k+1$. There exist computable positive constants $c_1$ and $c_2$ such that if the conditions
\begin{equation*}
\frac{n+1}{2}\sqrt{q} < c_1 |D| \text{\ \ and } \log_2 q \leq k < c_2 |D|
\end{equation*}
are satisfied, then $u$ is not a deep hole.
\end{thm}

\noindent Theorem \ref{dicksonTheorem2} shows that $D$ takes on a variety of sizes, large and small, depending on the parameters. This implies that progress is possible toward the deep hole problem for some small evaluation sets without any obvious algebraic structure to rely on.

\subsection{Example}

Take $q=2^{16}$ to consider a Reed-Solomon code over $\mathbb{F}_{2^{16}}$. When $n=3$, the Dickson polynomial family is $D_3(x,a)=x^3-3ax=x^3+ax$. Theorem \ref{dicksonTheorem2} gives $|D|=43691$. The conditions to satisfy are
\begin{equation*}
640 < 43691c_1 \text{\ \ and } 16 \leq k < 43691c_2
\end{equation*}
Fix $c_1=.015$. The proof shows that the largest choice of $c_2$ is
\begin{equation*}
c_2=65536^{-\frac{1}{k+1}}-\frac{1}{2}-.015
\end{equation*}
to examine
\begin{equation*}
16 \leq k < 43691\left( 65536^{-\frac{1}{k+1}}-\frac{1}{2}-.015 \right). 
\end{equation*}
The supported message sizes are $16\leq k\leq 21182$.

\subsection{Preliminaries}

\subsubsection{Weil's Character Sum Bound}

Our results rely on the following generalisation of Weil's classical character sum bound, see  \cite{FW}:

\begin{thm}
Let $f_i(t)$ ($1\leq i\leq n$) be polynomials in $\mathbb{F}_q[t]$, let $f_{n+1}(t)$ be a rational function in $\mathbb{F}_q(t)$, let $D_1$ be the degree of the highest square free divisor of $\prod_{i=1}^n f_i(t)$, let $D_2=0$ if $\deg(f_{n+1})\leq 0$ and $D_2=\deg(f_{n+1})$ if $\deg(f_{n+1})>0$, let $D_3$ be the degree of the denominator of $f_{n+1}$, and let $D_4$ be the degree of the highest square free divisor of the denominator of $f_{n+1}(t)$ which is relatively prime to $\prod_{i=1}^n f_i(t)$. Let $\chi_i:\mathbb{F}_q^* \to \mathbb{C}^*$ $(1\leq i\leq n)$ be multiplicative characters of $\mathbb{F}_q$, and let $\psi=\psi_p \circ \text{Tr}_{\mathbb{F}_q / \mathbb{F}_p}$ for a non-trivial additive character $\psi_p: \mathbb{F}_p \to \mathbb{C}^*$ of $\mathbb{F}_p$. Extend $\chi_i$ to $\mathbb{F}_q$ by setting $\chi_i(0)=0$. Suppose that $f_{n+1}(t)$ is not of the form $r(t)^p-r(t)+c$ in $\mathbb{F}_q(t)$. Then for any $m\geq 1$, we have
\begin{align*}
&\left| \sum_{a\in \mathbb{F}_{q^m}, f_{n+1}(a)\ne \infty} \chi_1(\text{N}_{\mathbb{F}_{q^m} / \mathbb{F}_q}(f_1(a))) \cdots \chi_n(\text{N}_{\mathbb{F}_{q^m} / \mathbb{F}_q}(f_n(a))) \psi(\text{Tr}_{\mathbb{F}_{q^m} / \mathbb{F}_q}(f_{n+1}(a))) \right| \\
&\phantom{\sum}\leq (D_1 + D_2 + D_3 + D_4 - 1)q^{m/2}
\end{align*}
where the sum is taken over those $a\in \mathbb{F}_{q^m}$ such that $f_{n+1}(a)$ is well-defined.
\end{thm}

\noindent We specify the parameters to obtain various character sum bounds.

\begin{cor}\label{weilBounds}
Let $\psi_{\text{Tr}}=\psi_p \circ \text{Tr}_{\mathbb{F}_q / \mathbb{F}_p}$ be as above, $\psi:\mathbb{F}_q\to \mathbb{C}^*$ a non-trivial additive character, and $\eta:\mathbb{F}_q^*\to \mathbb{C}^*$ the quadratic character if $q$ is odd. Set $m=1$.
\begin{enumerate}
\item If $f_{1}(x)=D_n(x,a)$ with $a\ne 0$:
\begin{equation*}
\left| \sum_{x\in \mathbb{F}_q} \psi(D_n(x,a))\right| \leq (n-1)\sqrt{q}
\end{equation*}
\item If $q$ is odd, $f_1(x)=x^2-4a$, and $f_{2}(x)=D_n(x,a)$ with $a\ne 0$:
\begin{equation*}
\left| \sum_{x\in \mathbb{F}_q} \eta(x^2-4a)\psi(D_n(x,a))\right| \leq (n+1)\sqrt{q}
\end{equation*}
\item If $q$ is even and $f_1(x)=b D_n(x,a)+a/x^2$ with $a,b\ne 0$:
\begin{align*}
\left|\sum_{x\in \mathbb{F}_q^*} \psi_{\text{Tr}}\left( b D_n(x,a)+a/x^2\right)\right| &=  \left| \sum_{x\in \mathbb{F}_q^*} \psi_{\text{Tr}}\left( b D_n(x,a)+a^{q/2}/x\right) \right| \\ 
&\leq (n+1)\sqrt{q}
\end{align*}
\end{enumerate}
\end{cor}

Note that none of the polynomials in place of $f_{n+1}(x)$ are of the form $r(t)^2-r(t)+c$. For instance, in part 3, such an $r$ would have to take the form $r(t)=d/x+f(x)$, where $f(x)$ is a polynomial. Expanding this shows that $d^2/x^2+d/x=a^{q/2}/x$, or $d=0$, which is a contradiction.

\begin{lem}\label{dicksonWeilEven}
Let $D=\{D_n(x,a)\ |\ x\in \mathbb{F}_q\}$, for $a\in \mathbb{F}_q^*$. If $\psi:(\mathbb{F}_q,+)\to \mathbb{C}^*$ is a non-trivial additive character, then the following estimate holds:
\begin{equation*}
\left|\sum_{x\in D} \psi(x) \right|\leq (n+1)\sqrt{q}. 
\end{equation*}
\end{lem}

\begin{proof}
The sum can be rewritten in the following way:
\begin{equation*}
\sum_{y\in D} \psi(y) = \sum_{x\in \mathbb{F}_q} \psi(D_n(x,a)) \frac{1}{N_x}, 
\end{equation*}
where $N_x=|D_n^{-1}(D_n(x,a))|$ is size of the preimage of the value $D_n(x,a)$. \newline

\noindent \textbf{When $q$ is even:}

\noindent By Theorem \ref{dicksonTheorem}, $N_x$ can be quantified. Let $\text{Tr}:\mathbb{F}_q\to \mathbb{F}_2$ denote the absolute trace. Using the fact that $z^2+xz+a$ is reducible over $\mathbb{F}_q$ if and only if $\text{Tr}(a/x^2)=0$,
\begin{align*}
&=\sum_{\substack{x\in \mathbb{F}_q^* \\ \text{Tr}(a/x^2)=0}} \frac{1}{\gcd(n, q-1)} \psi(D_n(x,a)) + \sum_{\substack{x\in \mathbb{F}_q^* \\ \text{Tr}(a/x^2)=1}} \frac{1}{\gcd(n, q+1)} \psi(D_n(x,a)) \\
&\phantom{=}\phantom{=} + \frac{1}{\gcd(n, q-1)}\psi(D_n(0,a)) + O(1), 
\end{align*}
where $O(1)$ is a constant of size at most 1, which we accept by dropping the $D_n(x,a)=0$ case. Denote $\psi_1: \mathbb{F}_2\to \mathbb{C}^*$ as the order two additive character and $\psi_{\text{Tr}}=\psi_1\circ \text{Tr}$, which is an additive character from $\mathbb{F}_q\to \mathbb{C}^*$. Simplifying and rearranging gives

\begin{align*}
&= \frac{1}{2\gcd(n,q-1)} \sum_{x\in \mathbb{F}_q^*} \psi(D_n(x,a))(1+\psi_{\text{Tr}}(a/x^2))  \\ &\phantom{=}\phantom{=} + \frac{1}{2\gcd(n, q+1)} \sum_{x\in \mathbb{F}_q^*} \psi(D_n(x,a))(1-\psi_{\text{Tr}}(a/x^2)) + \frac{1}{\gcd(n, q-1)}\psi(D_n(0,a)) + O(1) \\
&= \left( \frac{1}{2\gcd(n,q-1)} + \frac{1}{2\gcd(n, q+1)} \right) \sum_{x\in \mathbb{F}_q^*} \psi(D_n(x,a)) \\ &\phantom{=}\phantom{=} + \left( \frac{1}{2\gcd(n,q-1)} - \frac{1}{2\gcd(n, q+1)} \right) \sum_{x\in \mathbb{F}_q^*} \psi(D_n(x,a))\psi_{\text{Tr}}(a/x^2) \\
&\phantom{=}\phantom{=}\phantom{=}\phantom{=} + \frac{1}{\gcd(n, q-1)}\psi(D_n(0,a)) + O(1). 
\end{align*}
We add and subtract $\left(\frac{1}{2\gcd(n,q-1)} + \frac{1}{2\gcd(n, q+1)} \right)\psi(D_n(0,a))$ to complete the first sum:
\begin{align*}
&= \left( \frac{1}{2\gcd(n,q-1)} + \frac{1}{2\gcd(n, q+1)} \right) \sum_{x\in \mathbb{F}_q} \psi(D_n(x,a)) \\
&\phantom{=}\phantom{=} + \left( \frac{1}{2\gcd(n,q-1)} - \frac{1}{2\gcd(n, q+1)} \right) \sum_{x\in \mathbb{F}_q^*} \psi(D_n(x,a))\psi_{\text{Tr}}(a/x^2) \\
&\phantom{=}\phantom{=}\phantom{=}\phantom{=} + \left( \frac{1}{2\gcd(n,q-1)} -\frac{1}{2\gcd(n, q+1)} \right)\psi(D_n(0,a)) + O(1). 
\end{align*}
In order to estimate the sum in second term, take $b\in \mathbb{F}_q$ so that $\psi(x)=\psi_{\text{Tr}}(bx)$. Then,
\begin{equation*}
\sum_{x\in \mathbb{F}_q^*} \psi(D_n(x,a))\psi_{\text{Tr}}(a/x^2) = \sum_{x\in \mathbb{F}_q^*} \psi_{\text{Tr}}(bD_n(x,a)+a/x^2). 
\end{equation*}
Applying the bounds in Corollary \ref{weilBounds},
\begin{align*}
\left|\sum_{y\in D} \psi(y)\right| &\leq \left( \frac{1}{2\gcd(n,q-1)} + \frac{1}{2\gcd(n, q+1)} \right)(n-1)\sqrt{q} \\
&\phantom{=}\phantom{=}+ \left| \frac{1}{2\gcd(n,q-1)} - \frac{1}{2\gcd(n, q+1)} \right| (n+1)\sqrt{q} + 2 \\
&\leq (n+1)\sqrt{q}.
\end{align*}

\noindent \textbf{When $q$ is odd:}

\noindent We use Theorem \ref{dicksonTheorem} again to calculate $N_x$. Let $\eta$ be the quadratic character of $\mathbb{F}_q$.
\begin{equation*}
=\sum_{\substack{x\in \mathbb{F}_q \\ \eta(x^2-4a)=1}} \frac{1}{\gcd(n, q-1)} \psi(D_n(x,a)) + \sum_{\substack{x\in \mathbb{F}_q \\ \eta(x^2-4a)=-1}} \frac{1}{\gcd(n, q+1)} \psi(D_n(x,a)) + O(1). 
\end{equation*}
The term $O(1)$ is a constant of size at most 2, which we accept by dropping the complicated `condition C' and `otherwise' cases. Simplifying and rearranging gives
\begin{align*}
&= \frac{1}{2\gcd(n,q-1)} \sum_{x\in \mathbb{F}_q} \psi(D_n(x,a))(1+\eta(x^2-4a))  \\ &\phantom{=}\phantom{=} + \frac{1}{2\gcd(n, q+1)} \sum_{x\in \mathbb{F}_q} \psi(D_n(x,a))(1-\eta(x^2-4a)) + O(1) \\
&= \left( \frac{1}{2\gcd(n,q-1)} + \frac{1}{2\gcd(n, q+1)} \right) \sum_{x\in \mathbb{F}_q} \psi(D_n(x,a)) \\ &\phantom{=}\phantom{=} + \left( \frac{1}{2\gcd(n,q-1)} - \frac{1}{2\gcd(n, q+1)} \right) \sum_{x\in \mathbb{F}_q} \psi(D_n(x,a))\eta(x^2-4a) + O(1). 
\end{align*}
Again applying the bounds in Corollary \ref{weilBounds},
\begin{align*}
\left|\sum_{x\in D} \psi(x) \right|&\leq \left( \frac{1}{2\gcd(n,q-1)} + \frac{1}{2\gcd(n, q+1)} \right)(n-1)\sqrt{q} \\ &\phantom{=}\phantom{=}+ \left| \frac{1}{2\gcd(n,q-1)} - \frac{1}{2\gcd(n, q+1)} \right|(n+1)\sqrt{q} + 2\\
&\leq (n+1)\sqrt{q}, 
\end{align*}
which was to be shown.
\end{proof}

\subsubsection{Li-Wan's New Sieve}

We also state Li-Wan's new sieve (as in \cite{LW, WZ}) to estimate a particular type of character sum: let $D$ be a finite set and $D^k=D\times D\times \cdots \times D$ be the Cartesian product of $k$ copies of $D$. Let $X$ be a subset of $D^k$. Denote
\begin{equation*}
\overline{X} = \{(x_1, x_2,\ldots, x_k)\in X\ |\ x_i\ne x_j, i\ne j\}
\end{equation*}

\noindent Let $f(x_1, x_2,\ldots, x_k)$ be a complex-valued function defined over $X$. Denote
\begin{equation*}
F=\sum_{\mathbf{x}\in \overline{X}} f(x_1, x_2,\ldots, x_k)
\end{equation*}

\noindent Let $S_k$ be the symmetric group on $\{1, 2,\ldots, k\}$. Each permutation $\tau\in S_k$ can be uniquely factorised as a product of disjoint cycles and each fixed point is viewed as a trivial cycle of length 1. Namely,
\begin{equation*}
\tau=(i_1i_2\ldots i_{a_1})(j_1j_2\ldots j_{a_2})\cdots (l_1l_2\ldots l_{a_s})
\end{equation*}

\noindent with $a_i\geq 1$ and $1\leq i\leq s$. Define
\begin{equation*}
X_\tau = \{(x_1, x_2,\ldots, x_k)\ |\ x_{i_1}=\ldots=x_{i_{a_1}}, x_{j_1}=\ldots= x_{j_{a_2}}, \cdots, x_{l_1}=\ldots=x_{l_{a_s}} \}
\end{equation*}

\noindent Similarly define
\begin{equation*}
F_\tau= \sum_{\mathbf{x}\in X_\tau} f(x_1,x_2,\ldots, x_k)
\end{equation*}

\noindent We say that $\tau$ is of the type $(c_1, c_2,\ldots, c_k)$ if it has exactly $c_i$ cycles of length $i$. Let $N(c_1, c_2,\ldots, c_k)$ be the number of permutations of type $(c_1, c_2,\ldots, c_k)$. Define
\begin{equation*}
C_k(t_1, t_2,\ldots, t_k)=\sum_{\sum ic_i=k} N(c_1, c_2,\ldots, c_k)t_1^{c_1} t_2^{c_2}\cdots t_k^{c_k}
\end{equation*}

\noindent Now we have the following combinatorial result:
\begin{lem}
Suppose $q\geq d$. If $t_i=q$ for $d|i$ and $t_i=s$ for $d \nmid i$, then we have
\begin{align*}
C_k(s,\ldots s,q,s,\ldots,x,q,\ldots)&=k! \sum_{i=0}^{\lfloor k/d\rfloor} \binom{\frac{q-s}{d}+i-1}{i} \binom{s+k-di-1}{k-di} \\
&\leq \left(s+k+\frac{q-s}{d}-1\right)_k
\end{align*}
where $(x)_k=x(x-1)(x-2)\cdots (x-k+1)$.
\end{lem}

\noindent Furthermore, we say that $X$ is symmetric if for any $x\in X$ and any $g\in S_k$, we have $g\circ x\in X$. Also, if a complex-valued function $f$ is defined on $X$, we say that it is normal on $X$ if $X$ is symmetric and for any two conjugate elements in $S_k$, $\tau$ and $\tau'$, we have
\begin{equation*}
\sum_{x\in X_\tau} f(x_1, x_2,\ldots, x_k)=\sum_{x\in X_{\tau'}} f(x_1, x_2,\ldots, x_k)
\end{equation*}

\noindent Then, we have the result:
\begin{lem}
If $f$ is normal on $X$, then
\begin{equation*}
F=\sum_{\sum ic_i=k} (-1)^{k-\sum c_i} N(c_1, c_2,\ldots, c_k) F_\tau
\end{equation*}
\end{lem}

\subsubsection{A Rephrasing of Error Distance}

A simple argument shows that error distance can be rephrased in the following way:

\begin{lem}\label{errorD}
Let $\mathcal{C}$ be a Reed-Solomon code over $\mathbb{F}_q$ using the evaluation set $D$. Let $u$ be a received word and $u(x)$ be its interpolated polynomial with $\deg u(x)=k+1$. The error distance satisfies $d(u,\mathcal{C})\leq |D|-k-1$ if and only if there exists a subset $\{x_{i_1}, x_{i_2}, \ldots, x_{i_{k+1}}\}\subset D$ such that
\begin{equation*}
u(x)-v(x)=(x-x_{i_1})(x-x_{i_2})\cdots (x-x_{i_{k+1}})
\end{equation*}
for some $v(x)$ with $\deg v(x)\leq k-1$.
\end{lem}

\section{The Proof}

\begin{proof}
For the received word $u$, write the interpolated polynomial as
\begin{equation*}
u(x)=x^{k+1}-b_1 x^k + \ldots + (-1)^{k+1} b_{k+1}
\end{equation*}
By Lemma \ref{errorD}, if $u(x)$ is not a deep hole, then there exists a codeword $v(x)$ of degree $\leq k-1$ where
\begin{equation*}
u(x)-v(x) = (x-x_1)\cdots (x-x_k)(x-x_{k+1})
\end{equation*}
for distinct values of $x_1, \ldots, x_k, x_{k+1}$ in $D$. By expanding this product, we see that $u$ will not be a deep hole if and only if the equation
\begin{equation*}
x_1+\ldots+x_k+x_{k+1}=b_1
\end{equation*}
has a solution with distinct coordinates for every $b_1$ in $\mathbb{F}_q$. Let $N_u$ be the number of solutions to this equation and $G$ be the group of additive characters $\psi:\mathbb{F}_q\to \mathbb{C}^*$. By the orthogonality of characters,
\begin{equation*}
N_u=\frac{1}{q}\sum_{\substack{x_i\in D \\ \text{distinct}}} \sum_{\psi\in G} \psi(x_1+\ldots+x_k+x_{k+1}-b_1)
\end{equation*}
Removing the trivial character and exchanging the sums, we obtain 
$$
\left|N_u-\frac{1}{q}(|D|)_{k+1}\right| = \left|\frac{1}{q}  \sum_{\substack{\psi\in G \\ \psi\ne 1}} \psi(b_1)^{-1} \sum_{\substack{x_i\in D \\ \text{distinct}}}  \psi(x_1+\ldots+x_k+x_{k+1})\right| . 
$$
Now we can estimate the inner sum of the right hand side using Li-Wan's new sieve. Let $X=D^{k+1}$, $f(x_1,\ldots, x_k, x_{k+1}) = \psi(x_1 + \ldots + x_k + x_{k+1})=\psi(x_1)\cdots \psi(x_k) \psi(x_{k+1})$, so $F=\sum_{\mathbf{x}\in \overline{X}} f(\mathbf{x})$. $F$ is symmetric and normal. Then 
\begin{equation*}
F_\tau = \sum \psi(x_{11})\cdots \psi(x_{1c_1})\cdots \psi^{k+1}(x_{(k+1)c_1}) \cdots \psi^{k+1}(x_{(k+1) c_{k+1}})
\end{equation*}
 where the sum runs over $x_{s t_s}\in D$, $1\leq s\leq k+1$, and $1\leq t_s\leq c_s$. Applying Lemma \ref{dicksonWeilEven}, we obtain 
\begin{align*}
\left|N_u-\frac{1}{q}(|D|)_{k+1}\right| &\leq \left| \sum_{\substack{x_i\in D \\ \text{distinct}}}  \psi(x_1+\ldots+x_k+x_{k+1})\right| \\
&\leq \sum_{\sum ic_i=k+1} N(c_1, \ldots, c_{k+1}) |F_\tau| \\
&\leq C_{k+1}((n+1)\sqrt{q},|D|,(n+1)\sqrt{q},|D|,\ldots,(n+1)\sqrt{q},|D|) \\
&\leq \left( (n+1)\sqrt{q}+(k+1)+\frac{|D|-(n+1)\sqrt{q}}{2}-1\right)_{k+1} \\
&= \left( \frac{(n+1)\sqrt{q}}{2} + k + \frac{|D|}{2} \right)_{k+1}. 
\end{align*}
To guarantee that $N_u>0$, it suffices to have
\begin{equation*}
\frac{1}{q}(|D|)_{k+1} >  \left( \frac{(n+1)\sqrt{q}}{2} + k + \frac{|D|}{2} \right)_{k+1}, 
\end{equation*}
which is true if
\begin{equation*}
\frac{|D|}{\frac{n+1}{2}\sqrt{q}+k+\frac{|D|}{2}} > q^{\frac{1}{k+1}}. 
\end{equation*}
If we take $\dfrac{n+1}{2}\sqrt{q} < c_1 |D|$ and $k<c_2 |D|$ for positive constants $c_1$ and $c_2$, we calculate the condition
\begin{equation*}
q^{-\frac{1}{k+1}}-\frac{1}{2} > c_1+c_2. 
\end{equation*}
Therefore, it is enough to find
\begin{equation*}
q^{-\frac{1}{k+1}} > \frac{1}{2}. 
\end{equation*}
With some rearrangement, we have
\begin{equation*}
k > (\log_2 q) - 1, 
\end{equation*}
or more simply,
\begin{equation*}
k \geq \log_2 q. 
\end{equation*}
\end{proof}

\section{Conclusions}

We were able to show that words of degree $k+1$ were not deep holes in certain Reed-Solomon codes based on Dickson polynomials. We reduced to solving a type of restricted subset sum problem. If $D$ is the value set of $D_n(x,a)$ over some finite field $\mathbb{F}_q$, then this problem essentially asks to find conditions involving $q$, $n$, and a parameter $r$ such that
\begin{equation*}
x_1+\ldots+x_r=c
\end{equation*}
always has a solution using distinct $x_i\in D$, for any $c\in \mathbb{F}_q$. It is interesting to note that this is actually a harder version of Waring's problem for Dickson polynomials, which was originally studied without a distinctness condition by Gomez and Winterhof \cite{GW} and improved later by Ostafe and Shparlinski \cite{OS}. If we write $x_i=D_n(u_i,a)$ for suitable $u_i\in \mathbb{F}_q$, then our techniques give conditions to guarantee a solution to
\begin{equation*}
D_n(u_1,a) + \ldots + D_n(u_r,a)=c
\end{equation*}
where we require that each value $D_n(u_i,a)$ is distinct.

In a larger scope, this work was actually motivated by an attempt to use the stronger result from Keti-Zhu in \cite{KWZ}. As summarised in the introduction, it allows one to bound the error distance between a received word and code if the character sum bound
\begin{equation*}
\left|\sum_{b\in D} \chi(1-bx)\right|\leq K q^{1/2}
\end{equation*}
can be achieved for a suitable $K$, for any non-trivial  $\chi$ belonging to a particular character group. We leave this to a future investigation. 

\end{document}